\documentclass[a4paper, 11pt]{amsart}

\newtheorem{theorem}{Theorem}[section]
\newtheorem{lemma}[theorem]{Lemma}
\newtheorem{corollary}[theorem]{Corollary}

\theoremstyle{definition}

\theoremstyle{remark}
\newtheorem{remark}[theorem]{Remark}

\numberwithin{equation}{section}

\DeclareMathOperator{\image}{im}

\DeclareMathOperator{\codim}{codim}

\begin{document}

\title{Rigidity of elliptic genera for non-spin manifolds}


 \author{Michael Wiemeler}
 \address{Universit\"at M\"unster\\ Mathematisches Institut\\ Einsteinstra\ss{}e 62\\48149 M\"unster\\Germany}
 \email{wiemelerm@uni-muenster.de}




\date{\today}


\begin{abstract}
  We discuss the rigidity of elliptic genera for non-spin manifolds \(M\) with \(S^1\)-action.
  We show that if the universal covering of \(M\) is spin, then the universal elliptic genus of \(M\) is rigid.
  Moreover, we show that there is no condition which only depends on \(\pi_2(M)\) that guarantees the rigidity in the case that the universal covering of \(M\) is non-spin.
\end{abstract}

\maketitle

\section{Introduction}
\label{sec:introduction}

A \(\Lambda\)-genus is a ring homomorphism \(\varphi:\Omega_*^{SO}\rightarrow \Lambda\) where \(\Lambda\) is a \(\mathbb{C}\)-algebra and \(\Omega^{SO}_*\) is the oriented bordism ring.
For such a homomorphism one denotes by
\begin{equation*}
  g(u)=\sum_{i\geq 0} \frac{\varphi[\mathbb{C} P^{2i}]}{2i+1} u^{2i+1} \in \Lambda[[u]]
\end{equation*}
the logarithm of \(\varphi\).

A \(\Lambda\)-genus \(\varphi\) is called \emph{elliptic} if there are \(\delta,\epsilon\in\Lambda\) such that its logarithm is given by
\begin{equation*}
  g(u)=\int_0^u\frac{dz}{\sqrt{1-2\delta z^2+\epsilon z^4}}.
\end{equation*}
Examples of elliptic genera are the signature and the \(\hat{A}\)-genus.
For background material on elliptic genera see \cite{MR1189136} and \cite{MR970278}.

There is a universal elliptic genus \(\phi:\Omega^{SO}_*\rightarrow \mathbb{C}[[q]]\) which actually takes values in  \(\mathbb{Q}[[q]]\).
For an oriented even-dimensional manifold \(M\) the coefficients of \(\phi[M]\) are indices of the signature operator on \(M\) twisted with some vector bundles (for a precise definition see \cite[Section 2]{MR954493}).
In the following we are mainly concerned with this genus (and its equivariant refinement).
Therefore we call \(\phi\) \emph{the} elliptic genus.

If the oriented closed even-dimensional manifold \(M\) is acted on by the circle group \(S^1\), then \(\phi[M]\) can be refined to the equivariant elliptic genus \(\phi_{S^1}[M]\in R(S^1)[[q]]\otimes \mathbb{Q}\) of \(M\) where \(R(S^1)=\mathbb{Z}[\lambda,\lambda^{-1}]\) is the complex representation ring of \(S^1\).
Note that \(\phi_{S^1}[M]\) is an equivariant bordism invariant of \(M\).

So \(\phi_{S^1}[M]\) is of the form
\[\phi_{S^1}[M]=\sum_{n=0}^\infty\sum_{k\in \mathbb{Z}} a_{nk} \lambda^k q^n\]
with \(a_{nk}\in \mathbb{Q}\) such that for fixed \(n\geq 0\) there are at most finitely many \(k\in \mathbb{Z}\) with \(a_{nk}\neq 0\).

\(\phi_{S^1}[M]\) is called rigid if \(a_{nk}=0\) whenever \(k\neq 0\).
It has been conjectured by Witten \cite{MR885560} that equivariant elliptic genera of closed, spin manifolds with \(S^1\)-action are rigid. Here an oriented manifold is called spin if its second Stiefel--Whitney class vanishes.

Witten's conjecture was proved by Bott-Taubes \cite{MR954493}.
Alternative proofs have been given in \cite{MR998662}, \cite{MR1331972}, \cite{MR1722036} and \cite{MR4146575}.
Herrera-Herrera \cite{MR1979364} proposed a proof of the rigidity of elliptic genera of \(\pi_2\)-finite, oriented manifolds.
However, a mistake in the proof was found by Amann and Dessai \cite{MR2600123}.
They suggested that assuming \(\pi_2\)- and \(\pi_4\)-finiteness would be enough to come around this problem and to prove the rigidity (see \cite{MR2916045}).

The aim of this paper is to clarify Lemma 1 of \cite{MR1979364}. We find a counterexample to the original statement and obtain the same conclusion under additional hypothesis. We note that Amann and Dessai clarified the proof of Lemma 2 of \cite{MR1979364}.

Lemma 1 of \cite{MR1979364} claims that if \(M\) is an even-dimensional orientable manifold acted on by \(S^1\), then the \(H\)-fixed point components in \(M\) which contain \(S^1\)-fixed points are orientable for every subgroup \(H\subset S^1\).
We prove this claim under the extra assumption that the universal covering of \(M\) is a spin manifold.
This then implies the rigidity of the elliptic genera of closed, oriented \(S^1\)-manifolds \(M\) with spin universal cover.

We note that if the oriented, closed, connected manifold \(M\) has a finite cover \(\tilde{M}\) which is spin then the rigidity of \(\phi_{S^1}[M]\) can be deduced directly from Bott-Taubes's theorem. Indeed, by the Lefschetz fixed point formula we can assume that there are \(S^1\)-fixed points in \(M\).
This implies that the action of \(S^1\) on \(M\) lifts to an action on \(\tilde{M}\).
Moreover, since \(\tilde{M}\rightarrow M\) is an equivariant covering map with \(k<\infty\) sheets the equivariant elliptic genera of \(\tilde{M}\) and \(M\) are related as follows:
\[k\phi_{S^1}[M]=\phi_{S^1}[\tilde{M}].\]
Therefore the rigidity for \(M\) follows from the rigidity for \(\tilde{M}\).

However, we also note that there are manifolds with spin universal cover which do not admit any finite spin cover \cite{Mathoverflow}. Therefore we think it is worth stating and proving the following:

\begin{theorem}
  \label{sec:main}
  Let \(M\) be a closed, oriented, even-dimensional \(S^1\)-manifold whose universal covering is spin.
  Then the equivariant elliptic genus of \(M\) is rigid.
\end{theorem}

Using an argument which goes back to Hirzebruch and Slodowy \cite{MR1066570} one sees that this theorem implies the following generalization of a result of Atiyah and Hirzebruch \cite{MR0278334}.

\begin{corollary}
  \label{sec:introduction-2}
    Let \(M\) be a closed, connected, oriented manifold whose universal covering is spin such that \(S^1\) acts non-trivially on \(M\).
  Then the \(\hat{A}\)-genus of \(M\) vanishes.
\end{corollary}

In \cite{MR2600123} (see also \cite{amann2023addendum})  examples of oriented, closed, simply connected, effective \(S^1\)-manifolds \(M\) with \(\pi_2(M)=\mathbb{Z}_2\) and \(\hat{A}(M)\neq 0\) were given.
In view of the original claim in \cite{MR1979364} and Corollary~\ref{sec:introduction-2} one might ask if there is any condition on \(\pi_2\) of a closed, connected, oriented, effective \(S^1\)-manifold \(M\) with non-spin universal covering which guarantees the vanishing of \(\hat{A}(M)\).
Using equivariant surgery, we give the following negative answer to this question.

\begin{theorem}
 \label{sec:introduction-1}
  Let \(V\) be a group such that there is a closed, connected, oriented manifold \(P\) of dimension \(4n\geq 8\) with non-spin universal covering and \(\pi_2(P)\cong V\).
  Then there is a closed, connected, oriented, effective \(S^1\)-manifold \(M\) of dimension \(4n\) with \(\hat{A}(M)\neq 0\) and \(\pi_2(M)\cong V\).
\end{theorem}

To explain the proof of Theorem~\ref{sec:main} we have to recall the strategy of proof of Bott and Taubes \cite{MR954493}.

Let \(M\) be a closed, spin, even-dimensional \(S^1\)-manifold. Then the proof of the rigidity of \(\phi_{S^1}[M]\) in \cite{MR954493} has two geometric inputs:
\begin{enumerate}
\item For every subgroup \(H\subset S^1\), the fixed set \(M^H\) is orientable.
\item Orientations on \(M^H\) can be chosen in a specific way.
\end{enumerate}

The first property is established in Lemma 10.1 of \cite{MR954493} and also follows from a theorem of Edmonds \cite{MR603614}.
How to choose the orientations for \(M^H\) is explained on pages 157-158 of \cite{MR954493}.
That the choice of orientations discussed there is possible is guaranteed by Lemma 8.1 of \cite{MR954493} which is a special case of Lemma 9.3 of \cite{MR954493}.
These lemmas also give relations between the weights of the isotropy representations at the various \(S^1\)-fixed points in \(M\).

From these two properties Bott and Taubes deduce the rigidity of elliptic genera using the Lefschetz fixed point formula and complex analysis.
If \(M\) is not spin such that the \(S^1\)-action has these two properties, then one can also prove the rigidity for this \(S^1\)-action along the same lines.
As noted by Herrera and Herrera \cite{MR1979364}, to do so one only needs the orientability of those components of \(M^H\) which contain \(S^1\)-fixed points.

Our proof is completed by showing that this holds if the universal covering of \(M\) is spin.

This paper is structured as follows. We give a counterexample to Lemma 1 of \cite{MR1979364} in Section \ref{sec:count-lemma-1}.
In Section~\ref{sec:solution-problem} we then give a proof of this lemma under the extra assumption that the universal covering is spin. With this result we can complete the proof of Theorem~\ref{sec:main}.
In Section~\ref{sec:equivariant-surgery} we then prove a refined version of Theorem~\ref{sec:introduction-1}.
Then in Section~\ref{sec:another-example} we give another example which shows that there is also a gap in the proof of Lemma 2 in \cite{MR1979364}.
This gap is different from the mistake found by Amann and Dessai.

I would like to thank Manuel Amann and Anand Dessai for discussions on the subject of this paper.
I would also like to thank Matthias Wink for comments on an earlier version of this paper.
I thank an anonymous referee for comments which helped to improve the readability of this paper.

This research was funded by the Deutsche Forschungsgemeinschaft (DFG, German Research Foundation) under Germany's Excellence Strategy EXC 2044--390685587, Mathematics M\"unster: Dynamics--Geometry--Structure and within CRC 1442 Geometry: Deformations and Rigidity.

\section{An example of an orientable $S^1$-manifold with non-orientable singular strata}
\label{sec:count-lemma-1}

Lemma 1 of \cite{MR1979364} claims that if \(M\) is an orientable \(S^1\)-manifold of even dimension and \(F\subset M^H\) is a \(H\)-fixed point component for some \(H\subset S^1\) with \(F^{S^1}\neq \emptyset\) then \(F\) is orientable.
Here we give an example that shows that this claim is actually false.

For \(i=1,2\) let \(V_i\) be the irreducible \(S^1\)-representation of weight \(i\).
Let \[M_1=S(V_1^{k_1}\oplus V_2^{k_2}\oplus \mathbb{R})\] be the unit sphere in the \(S^1\)-representation \(V_1^{k_1}\oplus V_2^{k_2}\oplus \mathbb{R}\), where \(S^1\) acts trivially on \(\mathbb{R}\) and \(k_1\geq 1\), \(k_2\geq 2\).

Then \(F_1:=M_1^{\mathbb{Z}_2}=S(V_2^{k_2}\oplus \mathbb{R})\) is connected and contains both \(S^1\)-fixed points. Moreover, by the slice theorem, there is an orbit \(O_1\subset F_1\) which has an invariant neighborhood in \(M_1\) of the form
\[U_1=S^1\times_{\mathbb{Z}_2} D(\mathbb{R}_-^{2k_1}\oplus \mathbb{R}^{2k_2-1}),\]
where \(\mathbb{R}_-\) is the nontrivial irreducible \(\mathbb{Z}_2\)-representation and \(\mathbb{R}\) is the trivial one-dimensional \(\mathbb{Z}_2\)-representation. Here \(D(V)\) denotes the unit disc in the representation \(V\).

Moreover, let
\[M_2=S^1\times_{\mathbb{Z}_2}(P(\mathbb{R}^3\oplus \mathbb{R}_-)\times S(\mathbb{R}_-^{2k_1-1}\oplus \mathbb{R}^{2k_2-2})).\]
Here \(P(V)\) denotes the projectivication of the representation \(V\).
Note that \(M_2\) is an orientable \(S^1\)-manifold.

Moreover, the \(\mathbb{Z}_2\)-fixed point component \(F_2:=(S^1/\mathbb{Z}_2)\times P(\mathbb{R}^3)\times S(\mathbb{R}^{2k_2-2})\) in \(M_2\) is non-orientable and contains an orbit \(O_2\) with an invariant neighborhood in \(M_2\) of the form
\[U_2=S^1\times_{\mathbb{Z}_2} D(\mathbb{R}_-^{2k_1}\oplus \mathbb{R}^{2k_2-1}).\]

Hence we can glue \(M_1\setminus \dot{U}_1\) and \(M_2\setminus \dot{U}_2\) along \(\partial U_1=\partial U_2\) to get an orientable \(S^1\)-manifold \(M\).
The manifold \(M\) contains a \(\mathbb{Z}_2\)-fixed point component \(F\) which can be constructed from \(F_1\) and \(F_2\) by glueing \(F_1\setminus \dot{U}_1^{\mathbb{Z}_2}\) and \(F_2\setminus \dot{U}_2^{\mathbb{Z}_2}\) along \(\partial U_1^{\mathbb{Z}_2}=\partial U_2^{\mathbb{Z}_2}\).
Since \[U_1^{\mathbb{Z}_2}=U_2^{\mathbb{Z}_2}=(S^1/\mathbb{Z}_2)\times D(\mathbb{R}^{2k_2-1}),\]
\(F\) is the equivariant connected sum  of \(F_1\) and \(F_2\) at the principal orbits \(O_1\) and \(O_2\).
So \(F\) is non-orientable and contains all \(S^1\)-fixed points.

The problem in the proof of Lemma 1 of \cite{MR1979364} seems to be a false application of Lemma 9.1 of \cite{MR954493}.
This lemma gives a formula for the first Chern class \(c_1(E)\in H^2(S^2;\mathbb{Z})\) of an equivariant complex vector bundle \(E\) over \(S^2\) acted on effectively by \(S^1\) by rotation.
It says that
\begin{equation}
  \label{eq:1}
  c_1(E)[S^2]=m_N-m_S,
\end{equation}
where \(m_N\) and \(m_S\) are the sums of weights of the \(S^1\)-representations on the fibers of \(E\) over the two fixed points \(N,S\in S^2\) respectively.

In their proof, Herrera and Herrera construct spheres \(S^2\) with \(S^1\)-action and equivariant maps \(f:S^2\rightarrow M\) and apply the above formula to \(E=f^*TM\).
This bundle has a compatible complex structure by Lemma 9.2 of \cite{MR954493}.
However, the circle actions on these \(S^2\)'s are in general not effective.
They have principal isotropy group some \(\mathbb{Z}_k\subset S^1\), \(k\geq 1\).
In this situation the formula (\ref{eq:1}) has to be corrected by multiplying the left hand side with \(k\).

So if \(k\) is even the parity of the right hand side does not tell anything about the parity of the first Chern class \(c_1(E)\). But that is needed for the proof.

\section{Manifolds with universal covering that is spin}
\label{sec:solution-problem}

In this section we argue, that the proof of the rigidity of elliptic genera can be carried out under the assumption that the universal cover of the \(S^1\)-manifold is spin.
This gives the following theorem.
\begin{theorem}
  \label{sec:manif-with-univ-1}
  Let \(M\) be a closed, oriented, even-dimensional \(S^1\)-manifold such that its universal cover admits a spin structure. Then the equivariant elliptic genus of \(M\) is rigid.
\end{theorem}

As explained in the introduction, to prove the theorem, it suffices to show the following lemma.

\begin{lemma}
  \label{sec:manif-with-univ-3}
  Let \(M\) be an oriented, even-dimensional \(S^1\)-manifold whose universal covering is spin. Moreover, let \(F\subset M^H\) be a fixed point component of a subgroup \(H\subset S^1\) which contains \(S^1\)-fixed points.
  Then the following holds:
  \begin{itemize}
  \item \(F\) is orientable.
  \item  An orientation for \(F\) can be chosen as specified in \cite[p. 157-158]{MR954493}.
  \end{itemize}
\end{lemma}

The basic observation which makes the proof of this lemma possible is the following:

\begin{lemma}
\label{sec:manif-with-univ}
  An orientable manifold \(M\) has spin universal covering if and only if for every map \(f:S^2\rightarrow M\) we have \(f^*w_2(TM)=0\).
\end{lemma}

Here \(w_2(E)\in H^2(B;\mathbb{Z}_2)\) denotes the second Stiefel--Whitney class of the vector bundle \(E\rightarrow B\).
This lemma is well known (see \cite[p. 88-89]{MR1031992}).
But for the sake of completeness we also give a proof here.

\begin{proof}
  We may assume that \(M\) is connected.
  Let \(p:\tilde{M}\rightarrow M\) be the universal covering of \(M\).
  Then by the Hurewicz theorem there is an isomorphism \(\pi_2(\tilde{M})\rightarrow H_2(\tilde{M};\mathbb{Z})\).
  Moreover, by the universal coefficient theorem there is an isomorphism
  \[H^2(\tilde{M};\mathbb{Z}_2)\cong \hom (H_2(M;\mathbb{Z});\mathbb{Z}_2)\cong \hom(\pi_2(\tilde{M}),\mathbb{Z}_2).\]
  Therefore \(\tilde{M}\) is spin if and only if for every map \(f:S^2\rightarrow \tilde{M}\) we have
  \[0=f^*w_2(T\tilde{M})=f^*p^*w_2(TM)=(p\circ f)^*w_2(TM).\]
  Since \(p\) is a covering we also have that \(p_*:\pi_2(\tilde{M})\rightarrow \pi_2(M)\) is an isomorphism.
  Hence, the claim follows.
\end{proof}

\begin{proof}[Proof of Lemma~\ref{sec:manif-with-univ-3}]
Using Lemma~\ref{sec:manif-with-univ} one can modify the proof of Lemma 9.3 of \cite{MR954493} slightly to see that this lemma holds for the tangent bundle of \(M\). Therefore the second claim in Lemma \ref{sec:manif-with-univ-3} follows once the first claim is shown.
So we concentrate on the proof of the first claim.

  Let \(M\) be an oriented, even-dimensional \(S^1\)-manifold whose universal covering is spin.
  Let \(F\subset M^H\) be a fixed point component of \(H=\mathbb{Z}_k\subset S^1\) which contains an \(S^1\)-fixed point.
  We can assume without loss of generality that \(H\) is the principal isotropy group of the induced \(S^1\)-action on \(F\).
The irreducible real \(H\)-representations are given by
\[V_0,V_1,\dots,V_{[k/2]},\]
where \(V_0\) is the trivial one-dimensional representation, for \(0<l<k/2\) we have \(V_l=\mathbb{C}\) on which an element \(z\in H\subset S^1\subset \mathbb{C}\) acts by complex multiplication with \(z^l\).
If, moreover, \(k\) is even then \(V_{k/2}=\mathbb{R}\) on which a generator of \(H\) acts by multiplication with \(-1\).

Accordingly there is a equivariant splitting
\[TM|_F=E_0\oplus E_1\oplus\dots\oplus E_{[k/2]}\]
into \(S^1\)-invariant subbundles, such that for each \(x\in F\), \(E_i|_x\cong V_i^{l_i}\) as \(H\)-represen\-ta\-tions for some \(l_i\geq 0\) and \(i=0,\dots,[k/2]\).
Note that \(E_0=TF\). Moreover, each \(E_i\), \(0<i<k/2\), has an \(S^1\)-invariant complex structure.

So if \(k\) is odd, orientability of \(F\) follows from the orientability of \(M\).
Therefore assume that \(k\) is even.
As in the proof of Lemma 1 of \cite{MR1979364}, we construct an \(S^1\)-equivariant map
\[f:S^2\rightarrow F,\]
such that
\begin{enumerate}
\item \(S^1\) acts on \(S^2\) by rotation with principal isotropy group \(H\) and fixed points \(N,S\).
\item \(f(N)=f(S)\in F^{S^1}\),
\item \(f^*w_2(E_0\oplus E_{k/2})\neq 0\) if \(F\) is not orientable.
\end{enumerate}

To construct \(f\) we start with an equivariant embedding
\[g:S^1\times (S^1/H)\hookrightarrow F,\]
here \(S^1\) acts on the second factor of \(S^1\times (S^1/H)\) by multiplication.
By moving the orbit \(g(\{1\}\times (S^1/H))\) into an \(S^1\)-fixed point, we can homotope \(g\) to a equivariant map
\[g_1:S^1\times (S^1/H)\rightarrow F,\]
which is an embedding on \((S^1\times (S^1/H))\setminus (\{1\}\times (S^1/H))\) and collapses \(\{1\}\times (S^1/H)\) to an \(S^1\)-fixed point.
Therefore there is a map \(f:S^2\rightarrow F\) with the properties (1) and (2) and such that \(\image f=\image g_1=:A\).
Using Mayer-Vietoris sequences one sees that
\begin{align*}
  g_1^*:H^2(A;\mathbb{Z}_2)&\rightarrow H^2(S^1\times (S^1/H);\mathbb{Z}_2),& f^*:H^2(A;\mathbb{Z}_2)&\rightarrow H^2(S^2;\mathbb{Z}_2)
\end{align*}
are isomorphisms.
Hence it suffices to show that we can choose \(g\) such that \(g^*w_2(E_0\oplus E_{k/2})=g_1^*w_2(E_0\oplus E_{k/2})\neq 0\) if \(F\) is non-orientable.

Note that, since the \(S^1\)-action on \(M\) is orientation preserving, the codimension of all \(H'\)-fixed point components for \(H'\subset S^1\) is even.
Therefore the inclusion \(F_0\hookrightarrow F\) is one-connected.
Here \(F_0\) denotes the union of all principal orbits in \(F\).
Note that \(F_0\) is an open and dense subset of \(F\).
Hence, if \(F\) is non-orientable, then \(F_0\) is also non-orientable.
Note that \(\pi:F_0\rightarrow F_0/S^1\) is a principal \((S^1/H)\)-bundle.

Therefore we have \begin{equation}\label{eq:4}E_0|_{F_0}=TF_0=\pi^*(T(F_0/S^1))\oplus \mathbb{R}.\end{equation}
Hence, it follows that \(F_0/S^1\) is a non-orientable manifold.
Choose an embedding \(g':S^1\hookrightarrow F_0/S^1\) such that \({g'}^*w_1(T(F_0/S^1))=a\) is a generator of \(H^1(S^1;\mathbb{Z}_2)\).

Then the principal  \((S^1/H)\)-bundle \(F_0|_{g'(S^1)}\) is trivial.
Therefore \(g'\) lifts to an \(S^1\)-equivariant embedding
\[g:S^1\times (S^1/H)\hookrightarrow F_0.\]

It follows from (\ref{eq:4}) that
\begin{align*}
  g^*E_0=g^*TF_0=g^*\pi^* (T(F_0/S^1))\oplus \mathbb{R} =\pi^*{g'}^*(T(F_0/S^1))\oplus \mathbb{R}.
\end{align*}
Hence we have
\begin{align*}
  g^*w_1(E_0)&=a\in H^*(S^1\times (S^1/H);\mathbb{Z}_2)=\Lambda_{\mathbb{Z}_2}(a,b)&
                                                                                     g^*w_2(E_0)=0.
\end{align*}

Moreover, since \(S^1\) acts transitively on \(S^1/H\), we have \(g^*E_{k/2}=E\otimes \gamma\) where \(E\) is a vector bundle over the first factor \(S^1\) and \(\gamma=S^1\times_HV_{k/2}\).
Note that \(E_{k/2}\) and \(E\) have even rank since \(F\) has even codimension in \(M\).
Hence, we have \(w_1(E)=g^*w_1(E_{k/2})\) and \(g^*w_2(E_{k/2})=w_1(E)w_1(\gamma)=w_1(E)b\).

Since \(E_0\oplus E_{k/2}\) is orientable, we have \(w_1(E)=g^*w_1(E_{k/2})=g^*w_1(E_0)=a\) and \(g^*w_2(E_{k/2})=ab\).
Therefore we get \[g^*w_2(E_0\oplus E_{k/2})=g^*w_2(E_{k/2})+a^2=g^*w_2(E_{k/2})=ab\neq 0.\]
This completes the construction of \(f\).

Since \(E'=\bigoplus_{i=1}^{k/2-1}E_i\) has an \(S^1\)-equivariant complex structure, we deduce from (2) and Lemma 9.1 of \cite{MR954493} that \(f^*c_1(E')=0\) and hence \(f^*w_2(E')=0\).

By Lemma~\ref{sec:manif-with-univ}, it follows that
\[0=f^*w_2(TM)=f^*w_2(E')+f^*w_2(E_0\oplus E_{k/2})=f^*w_2(E_0\oplus E_{k/2}).\]
Hence, by (3), \(F\) is orientable and the claim is proved.
\end{proof}

\begin{remark}
  The examples of the manifolds \(M\) and \(M_2\) in Section~\ref{sec:count-lemma-1} show that the first claim in Lemma \ref{sec:manif-with-univ-3} becomes false when one removes one of the assumptions that the universal covering of \(M\) is spin or that \(F\) contains a fixed point.
\end{remark}

Using the arguments of \cite[Section 1.5]{MR1066570} or \cite[Section 1.4]{MR1979364} we get:

\begin{corollary}
  \label{sec:manif-with-univ-4}
    Let \(M\) be a closed, connected, oriented manifold whose universal covering is spin such that \(S^1\) acts non-trivially on \(M\).
  Then the \(\hat{A}\)-genus of \(M\) vanishes.
\end{corollary}

For a manifold \(M\), denote by \(h_{M2}:\pi_2(M)\rightarrow H_2(M;\mathbb{Z})\) the Hurewicz map which sends a homotopy class of maps \(f:S^2\rightarrow M\) to the homology class \(f_*[S^2]\), where \([S^2]\) denotes the fundamental class of \(S^2\).
Note that Lemma~\ref{sec:manif-with-univ} can be used to prove the following sufficient conditions for the universal covering being spin.

\begin{lemma}
  \label{sec:manif-with-univ-2}
  A connected, orientable manifold \(M\) has spin universal covering if one of the following three conditions holds:
  \begin{enumerate}
  \item \(\image h_{M2}\) is finite and of odd order.
  \item \(\image h_{M2}\) is finite and \(M\) is spin$^c$.
  \item \(M\) is spin$^c$ and \(c_1(M)\) is torsion.
  \end{enumerate}
\end{lemma}

Here a oriented manifold \(M\)  is called spin$^c$ if \(w_2(TM)\) is contained in the image of the natural map \(H^2(M;\mathbb{Z})\rightarrow H^2(M;\mathbb{Z}_2)\). Moreover, if this is the case then for every preimage \(b\in H^2(M;\mathbb{Z})\) of \(w_2(TM)\) there is a spin$^c$ structure on \(M\) with \(c_1(M)=b\) (see \cite[Appendix D]{MR1031992}).
 
\begin{proof}[Proof of Lemma \ref{sec:manif-with-univ-2}]
  Let \(f:S^2\rightarrow M\) be a map. By Lemma~\ref{sec:manif-with-univ}, we have to show that under any of the three conditions \(f^*w_2(TM)=0\).
  
  First assume that (3) holds.
  Then since \(H^2(S^2;\mathbb{Z})\) is torsion-free, we have \(0=f^*c_1(M)\equiv f^*w_2(TM) \mod 2\).

  Next assume that \(\image h_{M2}\) is finite. Let \(p:S^2\rightarrow S^2\) be a map of degree \(|\image h_{M2}|\). Then \(f\circ p\) is null homologous. Hence we have
  \[0=p^*f^*w_2(TM)=|\image h_{M2}|f^*w_2(TM).\]
  So, if \(|\image h_{M2}|\) is odd, the claim follows.
  
  If \(M\) is spin$^c$, then we can do the above calculation with \(c_1(M)\) instead of \(w_2(TM)\) and conclude that \(0=f^*c_1(M)\equiv f^*w_2(TM) \mod 2\).
\end{proof}

Lemma~\ref{sec:manif-with-univ-2} together with Theorem~\ref{sec:manif-with-univ-1} and Corollary~\ref{sec:manif-with-univ-4} gives the following corollary.

\begin{corollary}
  \label{sec:manif-with-univ-5}
  Let \(M\) be a closed, connected, oriented, even-dimensional \(S^1\)-manifold such that one of the following conditions is satisfied:
    \begin{enumerate}
  \item \(\image h_{M2}\) is finite and of odd order.
  \item \(\image h_{M2}\) is finite and \(M\) is spin$^c$.
  \item \(M\) is spin$^c$ and \(c_1(M)\) is torsion.
  \end{enumerate}
  Then the equivariant elliptic genus of \(M\) is rigid. Moreover, the \(\hat{A}\)-genus of \(M\) vanishes if the action is non-trivial.
\end{corollary}

Conditions (1) and (2) in this corollary are quite close to the original condition of Herrera and Herrera for the rigidity of elliptic genera. However, the example of Amann and Dessai \cite{MR2600123} and Corollary~\ref{sec:equivariant-surgery-1} below show that the conditions after the ``and'' can not be removed.

\section{Equivariant surgery}
\label{sec:equivariant-surgery}

In this section we prove a refined version of Theorem~\ref{sec:introduction-1} from the introduction.
To do so we have to set up some notations.

Let \(M\) be a connected manifold with universal covering \(p:\tilde{M}\rightarrow M\).
Then from the Hurewicz theorem we get an isomorphism
\[\pi_2(M)\cong\pi_2(\tilde{M})\cong H_2(\tilde{M};\mathbb{Z}).\]
Since \(\pi_1(M)\) is acting on \(\tilde{M}\), we get a \(\mathbb{Z}[\pi_1(M)]\)-module structure on \(\pi_2(M)\).
Here \(\mathbb{Z}[\pi_1(M)]\) is the group ring of \(\pi_1(M)\).
In the rest of this section this \(\mathbb{Z}[\pi_1(M)]\)-module structure on \(\pi_2(M)\) is assumed.

If \(\dim M\geq 4\) and \(M'\) is constructed from \(M\) by taking the connected sum of \(M\) with a simply connected manifold \(N\) in a point \(x\in M\), then the universal covering  \(\tilde{M}'\) of \(M'\) is constructed from \(\tilde{M}\) by taking the connected sum with copies of \(N\) in every point in \(p^{-1}(x)\).
In particular,
\[\pi_1(M')\cong \pi_1(M),\quad\quad\quad \pi_2(M')\cong \pi_2(M)\oplus \mathbb{Z}[\pi_1(M')]\otimes \pi_2(N).\]

Similarly, if \(\dim M\geq 6\) and \(M'\) is constructed from \(M\) by surgery on an embedded sphere \(S^2\subset M\), then \(p^{-1}(S^2)\) is the disjoint union of embedded spheres in \(\tilde{M}\) and \(\tilde{M}'\) is constructed from \(\tilde{M}\) by surgery  on all of these spheres.
In particular,
\[\pi_1(M')\cong \pi_1(M),\quad\quad\quad \pi_2(M')\cong \pi_2(M)/\mathbb{Z}[\pi_1(M')][\phi],\]
where \(\phi:S^2\rightarrow M\) is the inclusion of the embedded sphere on which we do surgery.

Now we can state and prove the main result of this section.

\begin{theorem}
  \label{sec:equivariant-surgery-3}
  Let \(G\) be a finitely presented group and \(V\) a \(\mathbb{Z}[G]\)-module such that there is a closed, connected, oriented manifold \(P\) of dimension \(4n\geq 8\) with non-spin universal covering and
  \begin{equation}
    \label{eq:3}
\pi_1(P)\cong G,\quad\quad\quad \pi_2(P)\cong V.
\end{equation}

  Then there exists a closed, connected, oriented, effective \(S^1\)-manifold \(M\) of dimension \(4n\) such that
  \begin{itemize}
  \item \(\pi_1(M)\cong G\), \(\pi_2(M)\cong V\).
  \item The universal covering of \(M\) is non-spin.
  \item \(M\) is equivariantly bordant to a linear \(S^1\)-action on \(\mathbb{C} P^{2n}\). In particular, \(\hat{A}(M)=\hat{A}(\mathbb{C} P^{2n})\neq 0\).
  \end{itemize}
\end{theorem}

\begin{proof}
  Recall first that if \(M\) is a manifold and \(c\in H^2(M;R)\) for some ring \(R\) then \(c\) induces a homomorphism
  \[\pi_2(M)\rightarrow R,\quad\quad\quad [f]\mapsto f^*c[S^2].\]
  We denote this homomorphism also by \(c\).
  Note that the homomorphism \(c\) is invariant under the action of \(\pi_1(M)\) on \(\pi_2(M)\).

  Next we recall some properties of \(S^1\)-manifolds and their orbit spaces.

  Let \(N\) be a connected effective \(S^1\)-manifold of dimension \(m\geq 7\) such that
  \begin{enumerate}
  \item For each non-trivial \(H\subset S^1\), \(\codim M^H \geq 4\).
  \item There are \(S^1\)-fixed points in \(N\).
  \end{enumerate}

  We denote by \(N_0\) the union of principal orbits in \(N\).
  Then \(N_0\) is an open dense subset of \(N\) and by (1) the inclusion \(\iota:N_0\hookrightarrow N\) is 3-connected.
  Moreover, by (2), the map
  \[\mathbb{Z}=\pi_1(S^1)\rightarrow \pi_1(N_0)\cong \pi_1(N)\]
  induced by the inclusion of an orbit is trivial.

  By the exact homotopy sequence for the fibration \(p:N_0\rightarrow N_0/S^1\) we see that there is a exact sequence
  \[0\rightarrow \pi_2(N_0)\stackrel{p_*}{\rightarrow} \pi_2(N_0/S^1)\stackrel{c_1(N_0)}{\rightarrow} \mathbb{Z}\rightarrow 0.\]
  Here we have identified the connecting homomorphism \(\pi_2(N_0/S^1)\rightarrow \pi_1(S^1)\) in the exact sequence with the homomorphism induced by the first Chern class \(c_1(N_0)\in H^2(N_0/S^1;\mathbb{Z})\) of the \(S^1\)-bundle \(N_0\rightarrow N_0/S^1\).
  Moreover, we see that \(p_*:\pi_1(N_0)\rightarrow \pi_1(N_0/S^1)\) is an isomorphism.

  Since \(p^*T(N_0/S^1)\) is stably isomorphic to \(TN_0\) we see that for \(a\in \pi_2(N_0)\) we have
  \[w_2(TN)(\iota_*a)=w_2(TN_0)(a)=w_2(T(N_0/S^1))(p_*a).\]

  Now let \(\phi:S^2\times D^{m-3}\hookrightarrow N_0/S^1\), such that \(c_1(N_0)([\phi|_{S^2\times\{0\}}])=0\).
  Then there is an equivariant embedding
  \[\tilde{\phi}:S^1\times S^2\times D^{m-3}\hookrightarrow N_0\hookrightarrow N,\]
  where \(S^1\) acts on the first factor by multiplication, covering \(\phi\).

  In this situation we can do equivariant surgery on \(\tilde{\phi}\), i.e. construct an \(S^1\)-manifold
  \[N'=(N\setminus \tilde{\phi}(S^1\times S^2\times \dot{D}^{m-3}))\cup_{S^1\times S^2\times S^{m-4}} S^1\times D^{3}\times S^{m-4}\]
  by equivariant gluing.
  On the level of orbit spaces \(N_0'/S^1\) is constructed from \(N_0/S^1\) by surgery on \(\phi\).
  Moreover, \(N'\) is equivariantly bordant to \(N\) and satisfies (1) and (2) from above.

  In particular we have \(\pi_1(N_0'/S^1)\cong \pi_1(N_0/S^1)\) and
  \[\pi_2(N_0'/S^1)\cong \pi_2(N_0/S^1)/\mathbb{Z}[\pi_1(N_0/S^1)][\phi|_{S^2\times \{0\}}].\]
  Moreover, \(c_1(N_0')\) and \(w_2(T(N_0'/S^1))\)  are the homomorphisms induced from \(c_1(N_0)\) and \(w_2(T(N_0/S^1))\)  on this quotient, respectively.

  After these general remarks we describe the construction of \(M\).
  Let \(N=\mathbb{C}P^{2n}\) equipped with a linear effective \(S^1\)-action such that (1) and (2) hold.
  Moreover, let \(P'\) be the \((4n-1)\)-manifold from Lemma \ref{sec:equivariant-surgery-2} below.
  Then (\ref{eq:3}) holds for \(P'\) and the universal covering of \(P'\) is non-spin.

  Let \(\phi_1: D^{4n-1}\hookrightarrow N_0/S^1\) and \(\phi_2:D^{4n-1}\hookrightarrow P'\) be embeddings. Then form the equivariant connected sum
  \[N'=(N\setminus p^{-1}(\phi_1(\dot{D}^{4n-1})))\cup_{S^{4n-2}\times S^1} (P'\setminus \phi_2(\dot{D}^{4n-1}))\times S^1.\]

  Then \(N'\) satisfies (1) and (2) and is equivariantly bordant to \(N\) since \(P'\times S^1\) is an equivariant boundary.
  On the level of orbit spaces \(N_0'/S^1\) is the connected sum of \(N_0/S^1\) and \(P'\).
  Therefore we have \(\pi_1(N_0'/S^1)=\pi_1(P')=G\) and, moreover
  \[\pi_2(N_0'/S^1)\cong \pi_2(P')\oplus \mathbb{Z}[G] \otimes \pi_2(N_0/S^1) =\pi_2(P')\oplus \mathbb{Z}[G] a \oplus \mathbb{Z}[G] b,\]
  where \(a,b\) is a \(\mathbb{Z}\)-basis of \(\pi_2(N_0/S^1)\) such that \(a\) is a generator of the image of \(\pi_2(N_0)\rightarrow \pi_2(N_0/S^1)\) and \(c_1(N_0)(b)=1\).
  Note that \(c_1(N_0')|_{\pi_2(P')}=0\) since the restriction of the \(S^1\)-bundle \(N_0'\rightarrow N_0'/S^1\) to \(P'\setminus  \phi_2(\dot{D}^{4n-1})\) is trivial.

  Let \(g_1,\dots,g_k\) be generators of \(G\). Then the \((1-g_i)b\), \(i=1,\dots,k\), are in the kernels of both \(c_1(N_0')\) and \(w_2(T(N_0'/S^1))\).
  So we can represent these homotopy classes by disjointly embedded spheres with trivial normal bundle.
  We do equivariant surgery on these spheres to get a \(S^1\)-manifold \(N''\) such that (1) and (2) hold for \(N''\) and \(\pi_1(N_0''/S^1)\cong G\) and
  \[\pi_2(N_0''/S^1)\cong V\oplus \mathbb{Z}[G]a\oplus \mathbb{Z} b.\]
  By Lemma~\ref{sec:manif-with-univ}, we can find \(c\in \pi_2(P')\cong V\) such that \(w_2(TP')(c)=w_2(T(N_0/S^1))(a)\).
  Then \(c_1(N_0'')\) and \(w_2(T(N_0''/S^1))\) are trivial on \(c+a\in \pi_2(N_0''/S^1)\).
  So we can represent \(c+a\) by an embedding of a sphere with trivial normal bundle and do equivariant surgery on this sphere to get an \(S^1\)-manifold \(M\).
  For this \(M\), (1) and (2) hold.
  Moreover, \(\pi_1(M_0/S^1)\cong G\), \(\pi_2(M_0/S^1)\cong V\oplus \mathbb{Z}b\).
  Since \(c_1(M_0)(b)=1\), and \(c_1(M_0)|_V=0\) we see from the exact homotopy sequence for the fibration \(M_0\rightarrow M_0/S^1\) that
  \(\pi_1(M_0)=G\) and \(\pi_2(M_0)=V\).
  Since \(M_0\rightarrow M\) is 3-connected the claim follows.
\end{proof}

From Theorem~\ref{sec:equivariant-surgery-3} we immediately get the following corollary which is Theorem~\ref{sec:introduction-1} from the introduction.

\begin{corollary}
  Let \(V\) be a group such that there is a closed, connected, oriented manifold \(P\) of dimension \(4n\geq 8\) with non-spin universal covering and \(\pi_2(P)\cong V\).
  Then there is a closed, connected, oriented \(S^1\)-manifold \(M\) of dimension \(4n\) with \(\hat{A}(M)\neq 0\) and \(\pi_2(M)\cong V\).
\end{corollary}

The following corollary shows that the condition (1) of Corollary~\ref{sec:manif-with-univ-5} gives the optimal condition which only depends on \(\image h_{M2}\) and guarantees the vanishing of the \(\hat{A}\)-genus of a closed, oriented \(S^1\)-manifold \(M\).

\begin{corollary}
  \label{sec:equivariant-surgery-1}
  Let \(V\) be a finitely generated, abelian group.
  Assume that \(V\) is either infinite or finite and of even order.
  Then there is a simply connected, closed, oriented manifold \(M\) of dimension \(4n\geq 8\) such that
  \begin{itemize}
  \item \(H_2(M;\mathbb{Z})=\pi_2(M)=V\).
  \item \(S^1\) acts effectively on \(M\).
  \item \(M\) is equivariantly bordant to a linear \(S^1\)-action on \(\mathbb{C}P^{2n}\). In particular, \(\hat{A}(M)=\hat{A}(\mathbb{C} P^{2n})\neq 0\).
  \end{itemize}
\end{corollary}

\begin{proof}
By Theorem~\ref{sec:equivariant-surgery-3} and the Hurewicz theorem it suffices to show that there is a closed oriented \(4n\)-manifold \(P\) which is simply connected and non-spin such that \(H_2(P;\mathbb{Z})=V\).

  To do so first note that by our assumption on \(V\) there is a non-trivial homomorphism \(\psi:V\rightarrow \mathbb{Z}_2\).
  Moreover, the group \(V\) fits into a exact sequence
  \begin{equation}
    \label{eq:2}
    0\rightarrow \mathbb{Z}^{k_1}\stackrel{\iota}{\rightarrow} \mathbb{Z}^{k_2}\stackrel{\varphi}{\rightarrow} V\rightarrow 0
  \end{equation}
  with some \(0\leq k_1\leq k_2<\infty\).
  Let \(a_{1},\dots,a_{k_1}\) be the standard basis of \(\mathbb{Z}^{k_1}\).

  Denote by \(W\) the total space of the non-trivial \(S^{4n-2}\)-bundle with structure group \(SO(4n-1)\) over \(S^2\).
  Then by taking the connected sum of several copies of \(W\) and \(S^2\times S^{4n-2}\), we construct a simply connected manifold \(P_1\) with \(H_2(P_1;\mathbb{Z})=\mathbb{Z}^{k_2}\) and such that \(w_2(TP_1)\in H^2(P_1;\mathbb{Z}_2)\) corresponds to \(\psi\circ \varphi\) under the isomorphism \(H^2(P_1;\mathbb{Z}_2)\cong \hom(H_2(P_1;\mathbb{Z}),\mathbb{Z}_2)\).

  Since (\ref{eq:2}) is exact we can represent the homology classes corresponding to the \(\iota(a_i)\) by disjointly embedded spheres in \(P_1\) with trivial normal bundles.
  By doing surgery on these spheres we get a manifold \(P\) as required.
\end{proof}

The following lemma is used in the proof of Theorem~\ref{sec:equivariant-surgery-3}.

\begin{lemma}
  \label{sec:equivariant-surgery-2}
  Let \(P\) be a closed, connected, oriented manifold of dimension \(n\geq 7\). Then there exists a closed, connected, oriented manifold \(P'\) such that
  \begin{itemize}
  \item \(\dim P'=n-1\).
  \item \(\pi_1(P')\cong \pi_1(P)\), \(\pi_2(P')\cong \pi_2(P)\).
  \item The universal covering of \(P'\) is spin if and only if the universal covering of \(P\) is spin.
  \end{itemize}
\end{lemma}
\begin{proof}
  Let \(f:P\rightarrow ]-1,n+1[\) be a Morse-function on \(P\) such that \(f(x)\in ](4i-1)/4,(4i+1)/4[\) for all critical points \(x\) of \(f\) with index \(i=0,\dots,n\).
  Let \(W=f^{-1}(]-1,7/2])\) and \(P'=f^{-1}(7/2)=\partial W\).
  Then \(P\) can be constructed from \(W\) by attaching handles of dimension at least \(4\).
  Similarly \(W\) can be constructed from \(f^{-1}([13/4,7/2])\cong P'\times [13/4,7/2]\) by attaching handles of dimension at least \(n-3\).
  Therefore the inclusion \(W\hookrightarrow P\) is \(3\)-connected.
  Moreover the inclusion \(P'\hookrightarrow W\) is \((n-4)\)-connected.
  Since \(n\geq 7\), it follows that \(P'\hookrightarrow P\) is \(3\)-connected.
  So the first two claims of the lemma follow.
  The last claim follows from Lemma \ref{sec:manif-with-univ} because \(TP'\oplus \mathbb{R}\cong TP|_{P'}\).
\end{proof}

\section{Another example}
\label{sec:another-example}

In this section we give an example which shows that there is also a gap in the proof of Lemma 2 in \cite{MR1979364} which persists even after the correction of Amann and Dessai.

This lemma is the analog of Lemma 9.3 in \cite{MR954493}.
Its purpose is to guarantee that the second claim of Lemma~\ref{sec:manif-with-univ-3} holds for \(\pi_2\)- and \(\pi_4\)-finite closed oriented \(S^1\)-manifolds \(M\).
In the proof of that lemma the finiteness assumption on the homotopy groups is only used to show that the parity of the sum of weights of the isotropy representation at any \(S^1\)-fixed point in \(M\) is independent of the fixed point.
Then the lemma is concluded from this condition on the weights.

However, the condition on the weights is not good enough to conclude the lemma as the following example shows.
Let \(\mathbb{C}P^2_\pm\) be \(\mathbb{C}P^2\) equipped with the following \(S^1\)-action
\[S^1\times \mathbb{C} P_\pm^2\rightarrow \mathbb{C}P_\pm^2\quad \quad (g,[z_0:z_1:z_2])\mapsto [z_0:g^1 z_1:g^{\pm 1} z_2].\] 

Then one can form the equivariant connected sum \(\mathbb{C}P^2_-\#\mathbb{C}P^2_+\) at the fixed points \([1:0:0]\in \mathbb{C} P^2_\pm\).
The resulting \(S^1\)-manifold \(M\) is diffeomorphic to \(\mathbb{C} P^2\#\mathbb{C} P^2\).
Moreover, the sum of weights at each \(S^1\)-fixed point in \(M\) is odd and all \(H\)-fixed point sets \(M^H\) for \(H\subset S^1\) are orientable.
However, Lemma 2 of \cite{MR1979364} does not hold for the \(\mathbb{Z}_2\)-fixed point component on which \(S^1\) acts non-trivially. And, indeed, \(\hat{A}(M)= -\frac{1}{4}\neq 0\).

We think the problem in the proof of Lemma 2 of \cite{MR1979364} is the last equation on page 350.
In this equation the constant \(c\) is not integral.
It is just a rational number with denominator equal to the order of the principal isotropy group of the \(S^1\)-action on the fixed point component in question.
Therefore one cannot say anything about the parity of the left hand side of the equation from the parity of the sum on the right hand side unless one only has odd order finite isotropy groups.

Under the assumption that all finite isotropy groups of points in \(M\) have odd order, the proofs of both Lemma 1 and Lemma 2 of \cite{MR1979364} work so that one can correct the main result of that paper as follows:

\begin{theorem}
  \label{sec:another-example-1}
  Let \(M\) be a connected, oriented, \(\pi_2\)- and \(\pi_4\)-finite closed manifold of even dimension such that \(S^1\) acts on \(M\) without finite isotropy groups of even order.
  Then the equivariant elliptic genus of \(M\) is rigid. Moreover, the \(\hat{A}\)-genus of \(M\) vanishes if the action is non-trivial.
\end{theorem}

Under the assumption that the \(S^1\)-action on \(M\) is semi-free, i.e. all orbits are either free orbits or \(S^1\)-fixed points, we get the following stronger result.

\begin{theorem}
  Let \(M\) be a connected, oriented, closed manifold of even dimension such that \(S^1\) acts semi-freely on \(M\).
  Assume that one of the following two conditions holds:
  \begin{itemize}
  \item \(M\) is \(\pi_2\)- and \(\pi_4\)-finite.
  \item The universal covering of \(M\) is spin.
  \end{itemize}
  Then a multiple of \(M\) is equivariantly bordant to a spin manifold with semi-free \(S^1\)-action.
\end{theorem}

\begin{proof}
  By \cite[Theorem VI]{MR246324}, \(\pi_2\)- and \(\pi_4\)-finiteness imply that the \(S^1\)-action is either even or odd, i.e. the codimensions of all fixed point components are either congruent to zero modulo four or to two modulo four, respectively.
  By the \((k=1)\)-case of the proof of Lemma 9.3 of \cite{MR954493} and Lemma \ref{sec:manif-with-univ}, the same conclusion holds if the universal covering of \(M\) is spin.

  By cutting out an open tubular neighborhood of \(M^{S^1}\) from \(M\),
  one sees that the normal unit sphere bundle to \(M^{S^1}\) bounds an oriented free \(S^1\)-manifold.
  Since the action on \(M\) is of even or odd type,
this means that the disjoint union of the normal unit sphere bundles to the fixed point components of codimension divisible by four (not divisible by four, respectively) bounds an oriented free \(S^1\)-manifold. 
  So the claim follows from
  Theorem 2.4 of \cite{MR882700}.
\end{proof}

We conjecture that under the assumptions of Theorem~\ref{sec:manif-with-univ-1} or of Theorem~\ref{sec:another-example-1}, a multiple of \(M\) is also equivariantly bordant to a spin manifold with  circle action.
We note in this context that the vanishing of the \(\hat{A}\)-genus implies that a multiple of \(M\) is non-equivariantly bordant to a spin manifold with non-trivial circle action (see \cite[Section 2.3]{MR0278334}).

\bibliography{counterexamples}{}
\bibliographystyle{amsplain}

\end{document}